\documentclass[11pt]{amsart}
\input epsf
\usepackage{latexsym}
\usepackage[usenames,dvipsnames]{color}
\usepackage[colorlinks=true, linkcolor=blue, citecolor=BrickRed, urlcolor=RoyalBlue]{hyperref}
            
\usepackage{amssymb, amsmath,amsthm,amscd, amssymb,
graphicx, enumerate, verbatim, mathrsfs, color, url, bm, hyperref, parskip}
\usepackage[all]{xy}\usepackage[labelformat=empty]{caption}

\usepackage[left]{lineno}

\numberwithin{equation}{section}
\newtheorem{cor}[equation]{Corollary}
\newtheorem{lem}[equation]{Lemma}
\newtheorem{prop}[equation]{Proposition}
\newtheorem{thm}[equation]{Theorem}

\newtheorem{Example}[equation]{Example}
\newenvironment{ex}{\begin{Example}\rm}{\end{Example}}
\newtheorem{remark}[equation]{Remark}
\newenvironment{rmk}{\begin{remark}\rm}{\end{remark}}
\def\co{\colon\thinspace}

\def\t{\mathfrak{t}}

\def\F{\mathbb{F}}

\def\a{\alpha}

\def\d{\partial}

\def\s{\sigma}

\def\Z{\mathbb{Z}}

\def\Z{\mathbb{Z}}
\def\S1{\bf S^1}

\makeatletter

\def\equalsfill{$\m@th\mathord=\mkern-7mu
\cleaders\hbox{$\!\mathord=\!$}\hfill
\mkern-7mu\mathord=$}
\makeatother

\begin{document}

\abovedisplayskip=6pt plus3pt minus3pt
\belowdisplayskip=6pt plus3pt minus3pt

\title[Four manifolds with no smooth spines]
{\bf Four manifolds with no smooth spines}

\thanks{\rm 2020 MSC\ Primary 57K40, 
Secondary 20F67, 53C20, 57K18, 57Q35, 57Q40.}
\thanks{Belegradek was partially supported by the Simons Foundation grant 524838.}

\author{Igor Belegradek}
\address{Igor Belegradek\\ School of Mathematics\\ Georgia Tech\\ Atlanta, GA, USA 30332\vspace*{-0.05in}}
\email{ib@math.gatech.edu}
\urladdr{www.math.gatech.edu/$\sim$ib}

\author{Beibei Liu}
\address{Beibei Liu\\ School of Mathematics\\ Georgia Tech\\ Atlanta, GA, USA 30332\vspace*{-0.05in}}
\email{bliu96@gatech.edu}
\urladdr{sites.google.com/view/beibei-liu/home/}

\begin{abstract}
Let $W$ be a compact smooth orientable $4$-manifold that deformation
retract to a {\scshape pl} embedded closed surface.
One can arrange the embedding to have at most 
one non-locally-flat point, and
near the point the topology of the embedding is encoded in the 
singularity knot $K$. If $K$ is slice, then $W$ 
has a smooth spine, i.e.,
deformation retracts onto a smoothly embedded surface.
Using the obstructions from the Heegaard Floer homology 
and the high-dimensional surgery theory,
we show that $W$ has no smooth spine if $K$
is a knot with nonzero Arf invariant, a nontrivial 
L-space knot, the connected sum of nontrivial L-space knots, or 
an alternating knot of signature $<-4$.
We also discuss examples where the interior of $W$ 
is negatively curved.
\end{abstract}
\maketitle
\thispagestyle{empty}

\vspace{-17pt}

\section{Introduction}

A {\em spine\,} is a topological
(not necessarily locally flat), compact, boundaryless 
submanifold that is a strong deformation retract of the ambient manifold.
A spine is {\em smooth or {\scshape pl}\,} 
if the submanifold has this property.

Examples of $4$-manifolds that are homotopy equivalent to closed
sufaces but have no {\scshape pl} spines can be found 
in~\cite{Mat75, MatVen, LevLid, HP}. It is shown in~\cite{Ven} that an example
in~\cite{MatVen} does not even have a topological spine.
Some $4$-manifolds with topological spines 
and no {\scshape pl} spines can be found in~\cite{RubKim}.
The present paper constructs $4$-manifolds with 
{\scshape pl} spines and no smooth spines.

In this section $W$ denotes a compact oriented 
smooth $4$-manifold with a {\scshape pl} spine $S$ 
homeomorphic to a closed oriented connected surface.
By a standard argument $S$ can be moved by
a {\scshape pl} homeomorphism to a spine 
with at most one non-locally-flat point; henceforth 
we assume that $S$ has this property. If $S$ is locally flat, 
then the submanifold $S$ is smoothable~\cite[Corollary 6.8]{RS-I}.
Otherwise, $S$ intersects the link of the non-locally-flat point in a
{\em singularity knot} $K$. If $K$ is smoothly slice, then replacing
the cone on $K$ in $S$ with a smoothly embedded disk in $W$ gives
a smooth spine of $W$.

Conversely,  if $\Sigma$ is
an oriented connected surface with one
boundary component, then attaching $\Sigma\times D^2$ to
the $4$-ball along the knot $K$ in its boundary with framing $r$
gives a compact oriented $4$-manifold
with a {\scshape pl} spine homeomorphic to $\Sigma/\d \Sigma$,
which has normal Euler number $r$ and singularity knot $K$.
If $\Sigma=D^2$, the $4$-manifold is denoted by $K^r$ and called a {\em knot trace}.

Examples of non-slice singularity knot such that $W$ has a smooth spine
come from exotic knot traces.
Namely~\cite[Theorem A]{Akb} describes 
knots $K_1$, $K_2$ such that $K_1$ is slice, $K_2$ is not slice,
and $K_1^r$, $K_2^r$ are diffeomorphic for some $r$.
We refer to~\cite{HMP, HP, FMNOPR} for a recent study of relations between
invariants of knot traces and knot concordance.

Cappell and Shaneson~\cite{CS76} developed a surgery-theoretic criterion
that can help decide when a manifold with a {\scshape pl} spine 
of dimension $\ge 3$ and codimension $2$ also admits a locally 
flat spine. Applying the criterion to $W\times S^1$ we prove

\begin{thm}
\label{thm: arf CS}
If $W$ is a compact oriented 
smooth $4$-manifold that has a {\scshape pl} spine
whose singularity knot has nonzero Arf invariant, then 
$W$ contains no smooth spine.
\end{thm}

Tye Lidman and Daniel Ruberman asked us if the generalized 
Rokhlin invariant can be used to give 
a purely $4$-dimensional proof of Theorem~\ref{thm: arf CS}.
This was done in~\cite[Theorem 3.1]{Sae} in the case when
$W$ is a homotopy $S^2$ with finite $H_1(\d W)$.
We leave the question to an interested reader. 

The criterion of~\cite{CS76} also 
gives a weak converse of Theorem~\ref{thm: arf CS}:
If $K$ has zero Arf invariant, then $W\times S^1$ has a smooth spine,
see Remark~\ref{rmk: arf reverse}.

If $W$ has two {\scshape pl} spines with regular neighborhoods
$R_1$, $R_2$ in the interior of $W$, then there is a homology cobordism 
between the boundaries $\d R_1$, $\d R_2$
obtained by gluing $W\setminus\mathrm{Int}(R_1)$ and
$W\setminus\mathrm{Int}(R_2)$ along $\d W$.
The Heegaard-Floer $d$-invariants $d_{\text{top}}$, $d_{\text{bot}}$
are preserved under homology cobordisms. Furthermore, one can express
the $d$-invariants of $\d R_1$, $\d R_2$ via singularity knots
of their spines, and for some knots the $d$-invariants
can be explicitly computed, which gives the following.
 
\begin{thm}
\label{thm: sing knots no smooth spine}
If $W$ is a compact oriented 
smooth $4$-manifold that has a {\scshape pl} spine
whose singularity knot is a nontrivial L-space knot, 
the nontrivial connected sum of nontrivial L-space knots, or an alternating knot of
signature $<-4$, then $W$ contains no smooth spine. 
\end{thm}

Recall that a knot $K\subset S^{3}$ is an \emph{L-space knot} if there is an integer $n>0$ such that the $n$-framed surgery on the knot is an L-space \cite[Definition 1.1]{OS05}.
For example, torus knots are L-space knots~\cite[p.1285]{OS05}.

The $d$-invariants obstruction applies to 
some topologically slice knots in~\cite{HKL}, which gives

\begin{cor}
\label{cor: top flat}
For any $g, e\in\Z$ with $g\ge 0$ there exists a compact smooth oriented
\mbox{${\textup 4}$-manifold}
with no smooth spine and 
a topological locally flat spine that is an oriented closed
genus $g$ surface with normal Euler number $e$.
\end{cor}

We were led to the subject of this paper while thinking of examples in~\cite{GLT, Kui} 
of oriented hyperbolic $4$-manifolds with {\scshape pl} spines.
Each of these manifolds is a quotient of the hyperbolic space 
$\mathbb H^4$ by a Kleinian group $\Gamma_0$ which is a torsion-free
finite index subgroup in a certain discrete group 
$\Gamma$ of orientation preserving isometries of
$\mathbb H^4$. The group $\Gamma$ is described via face-pairings of
its fundamental domain $F$, which
is obtained by removing from $\mathbb H^4$ a neighborhood of a nontrivial torus knot $T$
in the ideal boundary of $\mathbb H^4$.
In turn, the fundamental domain $F_0$ 
for $\Gamma_0$ is obtained by gluing $k$
copies of $F$, where $k$ is the index of $\Gamma_0$ in $\Gamma$, and one can describe
$F_0$ as the result of removing from $\mathbb H^4$ 
a neighborhood of the $k$-fold connected sum of $T$.
This $k$-fold connected sum is the singularity knot in a {\scshape pl} spine 
of $\mathbb H^4/\Gamma_0$, and hence $\mathbb H^4/\Gamma_0$ has no smooth spine
by Theorem~\ref{thm: sing knots no smooth spine}.

A related construction in~\cite[Section 6]{GLT} replaces the torus knot $T$ by an 
arbitrary nontrivial knot $K$ but the group $\Gamma$ is now generated by reflections
in the codimension one faces of $F$. The resulting singularity knot of the
{\scshape pl} spine of $\mathbb H^4/\Gamma_0$ is
the $\frac{k}{2}$-fold connected sum of $K\# r\bar K$,
where $r\bar K$ is the reverse of the mirror image of $K$.
(Here $k$ is even because $\mathbb H^4/\Gamma_0$ is orientable
and $\Gamma$ does not preserve orientation).
Since $K\# r\bar K$ is slice, the singularity knot is slice, 
and $\mathbb H^4/\Gamma_0$ has a smooth spine. 

An analog of these examples with variable pinched negative curvature is discussed
in~\cite{Bel-rn}, which is based on Ontaneda's Riemannian 
hyperbolization~\cite{Ont-ihes}. Here there is no need to pass to
a finite index torsion-free subgroup, and for any knot $K$
one gets pinched negatively curved
$4$-manifolds whose {\scshape pl} spine has $K$ as a singularity knot.
In particular, if $K$ satisfies the assumptions of Theorems~\ref{thm: arf CS} or 
\ref{thm: sing knots no smooth spine}, the negatively pinched $4$-manifold
has no smooth spine, while in the setting of 
Corollary~\ref{cor: top flat} there exists a topologically flat spine.
 
The structure of the paper is as follows. 
In Section~\ref{sec: kervaire} we review results on the Kervaire invariant
of compact oriented manifolds with codimension $2$ spines.
In Section~\ref{sec: kervaire and arf} we specialize to dimension $4$, relate
the Kervaire invariant of $W$ and the Arf invariant of the singularity knot, 
and prove Theorem~\ref{thm: arf CS}.
Section~\ref{sec: HF d-invariants} is a review of
Heegaard Floer $d$-invariants, whose relationship to $V$-functions
is explored in Section~\ref{sec: knots and d-invariants}.
In Section~\ref{sec: spines homology cobordisms}
we investigate how the assumption ``$W$ has a smooth spine''
affects the $V$-function of the singularity knot.
Section~\ref{sec: sing knots and spines} contains a proof of 
Theorem~\ref{thm: sing knots no smooth spine} and
Corollary~\ref{cor: top flat}.

{\bf Acknowledgments}: We are grateful to Lisa Piccirillo for leading us to~\cite{Akb}
and to Kyle Hayden for expository suggestions.
Liu appreciates the support and hospitality 
of the Max Planck Institute for Mathematics in Bonn, where she was a member when 
this work began. 

\section{Kervaire invariant of codimension two thickenings}
\label{sec: kervaire}

Let $W$ be a compact oriented {\scshape pl} manifold with a 
{\scshape pl} embedded spine $S$, a closed
connected oriented manifold of $\dim(S)=\dim(W)-2$.
Let $\xi$ be an oriented plane bundle over $S$ whose Euler class
is the normal Euler class of $S$ in $W$, and let 
$p_\xi\co D_\xi\to S$ be the associated $2$-disk bundle. 
Then~\cite[Proposition 1.6]{CS76} 
gives a homology isomorphism $h\co (W, \d W)\to (D_\xi, \d D_\xi)$
such that $h$ preserves the orientation class in the relative
second cohomology, and $p_\xi\circ h\vert_S$ is homotopic to the identity of $S$.
The map $h$ pulls $\a:=p_\xi^*(\nu_{\!_W}\vert_S)$ to 
the stable normal bundle $\nu_{\!_W}$ of $W$ because
$h^*\a$ and $\nu_{\!_W}$ are isomorphic over $S$ to which
$W$ deformation retracts. This gives a normal map $(h, b_W)$
where $b_W\co \nu_{\!_W}\to \a$ is the above bundle map.

Assuming, as we can, that $h$ is transverse regular to the zero section $S$
of $D_\xi$, we see that $N:=h^{-1}(S)$ is a closed surface which is 
locally flat in $W$ with normal bundle $h^*\xi$. 
The stable normal bundle to $N$ is 
$\nu_N=\nu_{\!_W}\vert_S\oplus h^*\xi=h^*(\a\vert_S\oplus\xi)$.
Thus $h\vert_N\co N\to S$ is covered by the bundle map 
$b_N\co\nu_N\to\a\vert_S\oplus\xi$.
The orientation on $\xi$ and $W$ defines an orientation on $N$ for which
$h\vert_N\co N\to S$ has degree one, and hence $(h\vert_N, b_N)$
is a normal map. 

The normal invariant of $(h\vert_N, b_N)$
is the image of the normal invariant of $(h, b_W)$
under the inclusion-induced map $[W, G/PL]\to [S, G/PL]$, which
is a bijection because $S\hookrightarrow W$ is a homotopy equivalence. 
This standard fact is stated on~\cite[p.195]{CS76} 
and in the appendix of~\cite{KimRub}, and 
the proof amounts to comparing various definitions
of the normal invariant.

By~\cite[Section 1]{RouSul} the Kervaire invariant of the normal map $(h\vert_N, b_N)$
is the Arf invariant of a certain quadratic form on 
the kernel of $h\vert_{N*}\co H_1(N;\Z_2)\to H_1(S; \Z_2)$.
A normal map with nontrivial Kervaire invariant represents a nontrivial class
in $[S, G/PL]$, see~\cite[Theorem 1.4(ii)]{RouSul}, and in fact,
the Kervaire invariant defines a group homomorphism
$[S, G/PL]\to\Z_2$~\cite[Corollary 4.5]{RouSul}. 

\section{Kervaire and Arf invariants in dimension four}
\label{sec: kervaire and arf}

Let us adopt notations of Section~\ref{sec: kervaire} and suppose $\dim(W)=4$.
Then the group $[S, G/PL]$ is isomorphic to $H^2(S;\Z_2)\cong\Z_2$, 
see e.g.~\cite[Section 2]{KirTay}, and hence the 
Kervaire invariant defines an isomorphism $[S, G/PL]\to\Z_2$.
 
Fix a triangulation of $W$ for which $S$ is a full subcomplex with only 
one non-locally flat point. Its star is an embedded 
$4$-ball $B$, and $C:=S\cap B$ is the cone on the knot $K=S\cap\d B$, 
the singularity knot of $S\subset W$. 

\begin{lem}
The Kervaire invariant of $W$ in 
$[S, G/PL]$ is the Arf invariant of the knot $K$.
\end{lem}
\begin{proof}
Let $V$ be the smallest subcomplex that contains a neighborhood of $S$ in $W$.
Since $S$ is a full subcomplex, $V$ is a regular neighborhood of $S$ in $W$,
to which $W$ deformation retracts.
Denote the relative 
interiors of $B$, $C$, $V$ by $\mathring{B}$, $\mathring{C}$, $\mathring{V}$.
Then $V\setminus\mathring{B}$ is a trivial $2$-disk bundle over 
$S\setminus\mathring{C}$. 
Give $B$ the structure of a trivial $2$-disk bundle over a $2$-disk, 
whose zero section $Z$ intersects $\d B$ in an unknot $U$.
Glue $V\setminus\mathring{B}$ and $B$ by an orientation-preserving 
$2$-disk bundle
automorphism identifying $\d C$ with $\d Z=U$
so that the resulting $2$-disk bundle $D_\xi\to S$
has the same Euler class as $S\subset W$. 
Denote the regular neighborhoods of $K$ and $U$ in $\partial B$ by
$R_K$ and $R_U$, respectively.

Then the above-mentioned map $h\co W\to D_\xi$
can be chosen so that $h\vert_{W\setminus\mathring{V}}$
is a deformation retraction onto $\d V$, the map  
$h\vert_{V\setminus B}$ is the identity,
$h$ takes $(B, \d B, K)$ to $(B, \d B, U)$, and 
maps $R_K$ homeomorphically onto $R_U$. 
To define $h\vert_B$ apply~\cite[Proposition 1.6]{CS76} to the thickening
$B$ of $C$ and use the fact that any homology equivalence 
$(R_K, \d R_K)\to (R_U, \d R_U)$ is homotopic to a homeomorphism.

Isotope the zero section $Z$ to a $2$-disk $Z_0\subset\d B$ rel boundary, and perturb
$h$ near $B$ to be transverse regular to $Z_0$. Then $\Sigma:=h^{-1}(Z_0)$
is a Seifert surface of $K$ and $N:=(S\setminus B)\cup\Sigma$ is a closed 
surface such that $h\co N\setminus\Sigma\to S\setminus Z_0$ is the identity. 
Since the surgery obstruction is additive~\cite[Theorem III.4.14]{Bro-book},
the Kervaire invariants of the normal maps  
$(h\vert_{\Sigma}, b_N\vert_\Sigma)$ and $(h\vert_N, b_N)$ are equal.
Finally, the Kervaire invariant of
$(h\vert_{\Sigma}, b_N\vert_\Sigma)$ equals the Arf invariant 
of $K$, as stated on~\cite[page XXXIII]{Ran-hd-knots} and proved in~\cite[Proposition 3.3]{Lev66}. 
\end{proof}

\begin{proof}[Proof of Theorem~\ref{thm: arf CS}]
The above thickening $V$ of $S$ is classified by the homotopy class
of a map $f\co S\to BSRN_2$. Let $\eta\co BSRN_2\to G/PL$
be the normal invariant map, see~\cite[p.182]{CS76}.
Then $\eta\circ f$ is the normal invariant of $S\hookrightarrow V$.
Since $K$ has nonzero Arf invariant, by the above discussion
the Kervaire invariant of $\eta\circ f$ is nonzero.

It is easy to check that the thickening $V^\prime:=V\times S^1$ of 
$S^\prime:=S\times S^1$ is the pullback of $S\hookrightarrow V$
under the coordinate projection $p\co S\times S^1\to S$, see~\cite[pp.173--175]{CS76}.
Let $i\co S\to S\times S^1$ be a section of $p$, say, given by $i(m)=(m,1)$.
Since $\eta\circ f=\eta\circ f\circ p\circ i$ is homotopically nontrivial,
so is $\eta\circ f\circ p$. Hence $S^\prime\hookrightarrow V^\prime$ 
is a thickening with nontrivial normal invariant.

Arguing by contradiction suppose that $W$ has a locally flat spine $L$. 
Then $L^\prime:=L\times S^1$ is a locally flat spine of $W^\prime$.
The restriction to $L^\prime$
of the deformation retraction $W^\prime\to S^\prime$ is homotopic to a diffeomorphism 
$g\co L^\prime\to S^\prime$, see e.g.~\cite[p.5]{Lau}. 
Hence the normal invariant of $g$ is trivial. 

As we explain in~\cite[Appendix C]{Bel-rn},
the pullback via $g$
of the Poincar{\'e} embedding given by the inclusion 
$S^\prime\subset W^\prime$
is isomorphic to the Poincar{\'e} embedding of the locally flat 
inclusion $L^\prime\subset W^\prime$.
Since $\dim(S^\prime)$ is odd and $\ge 3$, 
Theorem 6.2 of~\cite{CS76} implies that the Poincar{\'e} embedding for  
$L^\prime\subset W^\prime$ can be realized by a 
locally flat embedding if and only if the normal invariants
of $g$ equals the normal invariant of
the Poincar{\'e} embedding $S^\prime\subset W^\prime$.
This is a contradiction because these normal invariants are different and 
$L^\prime\subset W^\prime$ is locally flat. 
\end{proof}

\begin{rmk}
\label{rmk: arf reverse}
The above argument can be reversed, namely, if $K$ has zero Arf invariant,
then the Poincar{\'e} embedding  induced by the
inclusion $S\hookrightarrow W$ has trivial normal invariant, 
and hence so does its product with a circle 
or more generally, with any closed manifold $L$, 
and if $\dim(L)$ is odd, then $W\times L$
has a locally flat spine~\cite[Theorem 6.2]{CS76}.
\end{rmk}

\section{Heegaard Floer $d$-invariants and $V$-functions of knots}
\label{sec: HF d-invariants}

Ozsv{\'a}th and Szab{\'o} introduced~\cite{OS03, OS5, OS04} 
Heegaard-Floer homology theories $HF^{o}(M, \t)$ 
associated with a Spin$^{c}$ structure $\t$
on a closed oriented \mbox{$3$-manifold} $M$.
Here $o$ is a decoration indicating the flavor of a 
Heegaard-Floer theory, and in this paper
$o$ is $\infty$ or $-$.
The homology groups $HF^{-}(M, \t)$ and $HF^{\infty}(M, \t)$ 
are modules over $\Z[U]$ and $\Z[U, U^{-1}]$, respectively,
where $U$ is a formal variable whose action lowers the relative homological 
degree by $2$.
Related invariants for knots and links in $3$-manifolds 
were developed in~\cite{Ras, OS3, OS08}. We refer to these papers
for background.

Henceforth, we assume that $M$ has standard $HF^{\infty}$~\cite[p.240]{OS03},
and the Spin$^{c}$ structure $\t$ is torsion, i.e.,
its first Chern class has finite order in $H^2(M)$.

According to~\cite[Section 4.2.5]{OS5} 
the group $H_1^T(M):=H_{1}(M)/\text{Tors}$ acts
on the Heegaard-Floer chain complex $CF^{o}(M, \t)$, and on the corresponding
homology group $HF^{o}(M, \t)$. 
Let $HF^{o}(M, \t)_{\mathrm{bot}}$ and $HF^{o}(M, \t)_{\mathrm{top}}$ 
denote the kernel and the cokernel
of the $H_{1}^T(M)$-action on $HF^{o}(M, \t)$.
The \emph{$d$-invariants} $d_{\mathrm{top}}(M, \t)$ and $d_{\mathrm{bot}}(M, \t)$ are the maximal homological degrees of a non-torsion class
in  $HF^{-}_{\mathrm{top}}(M, \t)$ and $HF^{-}_{\mathrm{bot}}(M, \t)$, respectively,
see~\cite[Section 9]{OS03} and ~\cite[Section 3]{LR}. 
If $M$ is a rational homology sphere, the $H_{1}^T(M)$-action is trivial, so
that $HF^{-}_{\mathrm{top}}(M, \t)=HF^{-}_{\mathrm{bot}}(M, \t)=HF^{-}(M, \t)$, and $d_{\mathrm{top}}(M, \t)=d_{\mathrm{bot}}(M, \t)$ is the usual $d$-invariant for rational homology spheres, as in~\cite{OS03}. 
The invariants $d_{\mathrm{top}}(M, \t)$, $d_{\mathrm{bot}}(M, \t)$ are
preserved under rational homology cobordisms~\cite[Proposition 4.5]{LR}.

A null-homologous knot $K$ in 
$M$ gives rise to a $\mathbb{Z}\oplus \mathbb{Z}$ 
filtered chain complex $CFK^{\infty}(M, K, \mathfrak{s})$, 
which is a $\F[U, U^{-1}]$-module, see~\cite{OS3, Ras}. 
The filtration is indexed by the pair of integers 
$(i,j)$, where $i$ keeps track of the power of $U$, and $j$ records
the so called \emph{Alexander filtration}. For $s\in \mathbb{Z}$, 
let $A^{-}_{s}(K):=A^{-}_{s}(M, K, \mathfrak{s})$ be the subcomplex of 
$CFK^{\infty}(M, K, \mathfrak{s})$ corresponding to 
$\max(i, j-s)\le 0$~\cite[Remarks 3.7--3.8]{MO-links}. 
By the large surgery formula the homology of 
$A^{-}_{s}(K)$ is the sum of one copy of $\F[U]$ and a $U$-torsion submodule.

Following~\cite{NiWu} we define the {\em $V$-function\,}  $V_{s}(K)$
of an oriented knot $K\subset S^{3}$ 
so that $-2V_{s}(K)$ 
is the maximal homological degree of the free part of $H_{\ast}(A^{-}_s(K))$.
For one-component links the $H$-function for links of~\cite{BG, Liu} is 
the $V$-function of knots. 
For example, the $V$-function of the unknot $U$ is given by
$V_{s}(U)=0$ for $s\geq 0$ and $V_{s}(U)=-s$ for $s<0$.
The $V$-function takes values in
nonnegative integers~\cite[Proposition 3.10]{BG}, and furthermore,
\cite[Proposition 3.10]{BG} and~\cite[Lemma 5.5]{Liu} give

\begin{prop}
\label{Vproperty} The $V$-function of an oriented knot
$K\subset S^{3}$ satisfies
$$V_{-s}(K)=V_{s}(K)+s\qquad\text{and}\qquad 
V_{s-1}(K)-V_{s}(K)\in \{0, 1\}.$$
\end{prop}

\begin{ex}
\label{ex:alternating}
Let $K$ be an alternating knot of signature $\sigma$; recall that $\s\in 2\mathbb Z$. 
By~\cite[Theorem 1.7]{HM} if $\sigma>0$, then $V_{s}(K)=0$ for all $s$, and
if $\sigma\leq 0$,  the values $V_{0}(K)$ are given in the table below:

\begin{center}
\begin{tabular}{ |c|c|c| } 
 \hline
 $\sigma$   & $V_{0}(K)$\\ \hline 
$-8k$   & $2k$ \\ \hline
$-8k-2$  & $2k+1$ \\ \hline
$-8k-4$ &  $2k+1$ \\ \hline
$-8k-6$  & $2k+2$ \\ 
 \hline
\end{tabular}
\end{center}
\end{ex}

\section{Surgeries on knots and $d$-invariants}
\label{sec: knots and d-invariants}

For a positive integer $g$ let 
$C^g:=\#^{2g} S^{2}\times S^{1}$, the connected sum of $2g$ copies of $S^2\times S^1$.
As usual $M_n(K)$ denotes the $n$-framed surgery on a closed oriented $3$-manifold $M$
along a knot $K\subset M$; in what follows $M$ is $S^3$ or $C^g$.
 
If $B\subset C^g$ is the Borromean knot, as defined e.g. in~\cite[Figure 4.1]{Park}, then
$C^g_n(B)$ has the structure of an 
oriented circle bundle over the genus $g$ oriented surface with 
Euler number $n$~\cite[Section 5.2]{OS4}.
For the unknot $U\subset S^3$ it is well-known that $S^3_n(U)$ is 
an  oriented circle bundle over $S^2$ with Euler number $n$.

It follows from~\cite[Propositions 9.3--9.4]{OS03}, 
cf.~\cite[Proposition 4.0.5]{Park}, that 
$C^g_n(K\# B)$ has standard $HF^\infty$ for any oriented knot 
$K\subset S^3$. 
The same is true for $S^3_n(K)$~\cite[Theorem 10.1]{OS04}. 
Thus the $d$-invariants $d_{\mathrm{top}}$, $d_{\mathrm{bot}}$
are defined for $C^g_n(K\# B)$ and $S^3_n(K)$, and moreover, 
for $S^3_n(K)$ they reduce to the usual \mbox{$d$-invariants}.
They were computed by Ni-Wu~\cite[Proposition 1.6]{NiWu} for $S^3_n(K)$, and
by Park~\cite[Theorem 4.2.3]{Park} for $C^g_n(B)$, $n\neq 0$.
Park's argument extends to $C^g_n(K\# B)$ as follows.

\begin{thm}
\label{thm:surgeryfor}
For $n>0$, we have
\begin{equation}
d_{\mathrm{top}}(C^g_{n}(K\#B), k))=
g+\dfrac{(2k-n)^{2}-n}{4n}-2\min_{_{a=0, \cdots, g}} 
\!\!\{a+V_{k-g+2a}(K)\}.
\end{equation}
\begin{equation}
d_{\mathrm{bot}}(C^g_{n}(K\# B, k))= 
g +\dfrac{(2k-n)^{2}-n}{4n}-2\max_{_{a=0, \cdots, g}}\!\!\{a+V_{k-g+2a}(K)\}
\end{equation}
where $k$ labels the torsion Spin$^{c}$ structures on $C^g_{n}(K\# B)$ 
with $-n/2< k \leq n/2$. 
The $d$-invariant of $S^3_n(K)$ is given by
\begin{equation}
d(S^{3}_{n}(K), k)=\dfrac{(2k-n)^{2}-n}{4n}-2V_{k}(K).
\end{equation}
\end{thm}
\begin{proof}
As in the proof of~\cite[Theorem 4.1.1]{Park},
a diagram chase in the surgery mapping cone formula~\cite[Theorem 4.10]{OS4}
shows that the free part of  $H_{\ast}(A^{-}_{k})$ is isomorphic 
to the free part of $HF^{-}(C^g_{n}(K\# B), k)$.
The grading of the free part of $H_{\ast}(A^{-}_{k})$ can be found 
in~\cite[Theorem 6.10]{BHL}.
\end{proof}

\begin{rmk} 
Theorem~\ref{thm:surgeryfor} extends to $n<0$ as follows.
Since the Borromean knot is amphichiral, $C^g_{n}(K\# B)=-C^g_{-n}(\bar{K}\# B)$,
where $\bar{K}$ is the  mirror of $K$. Then~\cite[Proposition 3.7]{LR} gives 
\begin{align*}
d_{\mathrm{bot}}(C^g_{n}(K \# B), k)&=-d_{\mathrm{top}}(C^g_{-n}(\bar{K}\# B), k)
\\ 
d_{\mathrm{top}}(C^g_{n}(K \# B), k)&=-d_{\mathrm{bot}}(C^g_{-n}(\bar{K} \# B), k).
\end{align*}
\end{rmk}

\begin{rmk}
A similar argument also computes 
$d_{\mathrm{top}}$ and $d_{\mathrm{bot}}$ for rational surgeries, 
i.e., when $0\neq n\in\mathbb Q$. 
\end{rmk}

\section{Spines, homology cobordisms, and $d$-invariants}
\label{sec: spines homology cobordisms}

Let $W$ be a compact, oriented, smooth $4$-manifold with a {\scshape pl} spine $S_1$,
an oriented genus $g$ surface with normal Euler number $e$. 
As before assume that $S_1$ has at most one 
non-locally-flat point with singularity knot $K\subset S^3$. 
If $W$ also has a smooth spine $S_{2}$, then there is a homology cobordism  
$C$ between the boundaries $M_1$, $M_2$ of the regular neighborhoods of $S_{1}$, 
$S_{2}$. Namely, $C$ is obtained by removing the interiors
of the regular neighborhoods from $W$ and gluing the results along $\d W$. 

Here $M_{1}$ can be described as 
an $e$-surgery on $C^g=\#^{2g} S^{2}\times S^{1}$ along the knot $K\# B$ where 
$B$ is the Borromean knot~\cite[Theorem 3.1]{BHL}, while $M_{2}$ 
is the circle bundle over $S_{2}$ with Euler number $e$, 
which is the $e$-surgery on $C^g$ along $B$. 

Since $H^{2}(C)\cong\mathbb{Z}^{2g}\oplus\mathbb{Z}/e\mathbb Z$,
every torsion Spin$^{c}$ structure on $C$ can be thought of 
an element of $\mathbb{Z}/e\mathbb Z$ indexed by $k\in (-e/2, e/2]$.
Restricting the element to $M_j$, $j\in\{1, 2\}$, 
gives a torsion Spin$^{c}$ structure on $M_j$, which we denote $\t_{kj}$.
Thus
\begin{equation}
\label{form: dtop}
d_{\mathrm{top}}(M_{1}, \mathfrak{t}_{k1})= d_{\mathrm{top}}(M_{2}, \mathfrak{t}_{k2}). 
\end{equation}

\begin{thm}
\label{thm: smooth spine g over 2}
If $e\geq 0$, and $W$ contains a smooth spine, then the singularity knot $K$ satisfies
\begin{equation}
\label{restriction}
  \min_{_{a=0, \cdots, g}}\{a+ V_{-g+2a}(K)\}=\lceil g/2 \rceil,
\end{equation}
where $\lceil g/2 \rceil$ is the smallest integer that is $\ge g/2$.
\end{thm}
\begin{proof}
Since $V_{s}(U)=\dfrac{|s|-s}{2}$ we compute 
\begin{equation}
\label{form: unknot}
\min_{_{a=0, \cdots, g}}\!\!\{a+V_{-g+2a}(U)\}=\lceil g/2 \rceil.
\end{equation}
If $e>0$, by Theorem \ref{thm:surgeryfor}
$$d_{\mathrm{top}}(M_{1}, \t_{k1})=g+s-2\min_{_{a=0, \cdots, g}}\!\!\{a+V_{k-g+2a}(K)\}$$
 and
$$d_{\mathrm{top}}(M_{2}, \t_{k2})=g+s-2\min_{_{a=0, \cdots, g}}\!\!\{a+V_{k-g+2a}(U)\}$$
where $s=\frac{(2k-e)^{2}-e}{4e}$.
Hence 
\begin{equation}
\label{form: min}
\min_{_{a=0, \cdots, g}}\!\!\{a+V_{k-g+2a}(K)\}=
\min_{_{a=0, \cdots, g}}\!\! \{a+V_{k-g+2a}(U)\}.
\end{equation}
Combining (\ref{form: unknot}) and (\ref{form: min}) for $k=0$ gives (\ref{restriction})
in the case $e>0$. 

Assume $e=0$. Then $M_{j}$ is the $0$-surgery on $C^g$, where $j=1, 2$. 
Let $M'_{j}$ denote the $1$-surgery on $C^g$ along the same knot as for $M_{j}$. 
By the equality part of~\cite[Corollary 9.14]{OS03}, 
$$d_{\mathrm{top}}(M_{j}, \t_{0j})-\dfrac{1}{2}=d_{\mathrm{top}}(M'_{j}, \t'_{0j}),$$
where $\t_{0j}, \t_{0j}^\prime$ are the trivial Spin$^{c}$ structures. 
Even though~\cite[Corollary 9.14]{OS03} is stated for knots in $S^3$, it 
generalizes (with the same proof)
to knots in $3$-manifolds with standard $HF^\infty$ and trivial $HF_{\mathrm{red}}$,
which is how we apply it.

By (\ref{form: dtop}) $M_1$, $M_2$ have the same $d_{\mathrm{top}}$, and hence 
$d_{\mathrm{top}}(M'_{1}, \t'_{01})=d_{\mathrm{top}}(M'_{2}, \t'_{02})$,
and as before (\ref{form: unknot})--(\ref{form: min}) imply 
(\ref{restriction}), now for $e=0$.
\end{proof}

\begin{cor}
\label{cor: 0 or 1}
If $W$ contains a smooth spine with normal
Euler number $e\geq 0$, then the singularity  knot $K$ satisfies
\begin{equation}
\label{restriction2}
V_{0}(K)=0 \textup{\ if $g$ is even and}\  V_{1}(K)=0 \textup{ if $g$ is odd.}
\end{equation}
\end{cor}
\begin{proof}
If $g=2k$, then by Theorem~\ref{thm: smooth spine g over 2} 
$$\min_k\{ V_{0}(K)+k, V_{2}(K)+k+1, \cdots, V_{2k}(K)+2k\}=k.$$
Proposition~\ref{Vproperty} gives $V_{s-1}(K)\le V_s(K)+1$, and  
hence the minimum occurs for $V_{0}(K)+k=k$, 
which implies $V_{0}(K)=0$. Similarly, if $g=2k+1$, 
we have 
$$\min_k \{ V_{1}(k)+k+1, \cdots, V_{2k+1}(K)+2k+1\}=k+1$$
which means that $V_{1}(K)+k+1=k+1$, and hence $V_{1}(K)=0$.
\end{proof}

\section{Singularity knots and smooth spines}
\label{sec: sing knots and spines}

As in Section~\ref{sec: spines homology cobordisms}
let $W$ a compact, oriented, smooth $4$-manifold with a {\scshape pl} spine which is
an oriented genus $g$ surface with normal Euler number $e$,
and at most one non-locally-flat point with singularity knot $K$. 
After changing the orientation of $W$, if needed, 
we can and will assume that $e\ge 0$.

\begin{proof}[Proof of Theorem~\ref{thm: sing knots no smooth spine}]
By Corollary~\ref{cor: 0 or 1} $V_{0}(K)=0$ or $V_{1}(K)=0$ depending on the parity of $g$,
and hence $g(K)\le 1$~\cite[Lemma 2.11]{Liu1}, where $g(K)$ is the genus of 
of $K$. If $g(K)=0$, then $K$ is the unknot. A genus-one L-space knot is 
the right-handed trefoil~\cite[Corollary 1.5]{Ghi}. 
According to~\cite{NNU} the Arf invariant for the torus knot $T(p,q)$ is 
$(p^{2}-1)(q^{2}-1)/24\ (\textup{mod}\,2)$.
Thus the Arf invariant of $T(2, 3)$ is nonzero, which implies by 
Theorem~\ref{thm: arf CS} that $W$ cannot contain a smooth spine.  
This completes the proof when $K$ is an L-space knot. 

Suppose $K$ is an alternating knot of signature $<-4$. 
Hence $V_{0}(K)\geq 2$ by Example \ref{ex:alternating}. 
Then Proposition~\ref{Vproperty} gives $V_{1}(K)\geq 1$, which 
by Corollary~\ref{cor: 0 or 1} shows that $W$ does not have a smooth spine. 

Finally, suppose that $K$ is the connected sum of 
nontrivial L-space knots $K_1, \cdots, K_n$ with $n\ge 2$. 
Thus $g(K)=g(K_1)+\cdots+g(K_{n})$.
Since $K_i$ is nontrivial, $g(K_i)\ge 1$, and hence $g(K)\ge n$.
For $j\in\mathbb Z$ set $R_{K_i}(j):=V_{g(K_i)-j}(K_i)$ and
$$R_{K}(j):=\min_{_{j_{1}+\cdots +j_{n}=j}} R_{K_{1}}(j_{1})+\cdots +R_{K_{n}}(j_{n}).$$
By Proposition~\ref{Vproperty} the function $R_{K_{i}}$ is nonnegative and
nondecreasing, and combining the proposition with~\cite[Lemma 2.11]{Liu1} gives 
$R_{K_{i}}(1)=V_{g(K_{i})-1}(K_i)=1$. Hence $R_{K_{i}}(j)\ge 1$ for every $j\ge 1$.

Propositions 5.1 and 5.6 and Lemma 6.2 of~\cite{BorLiv} imply 
$V_{j}(K)+j=R_{K}(g(K)+j)$; the notations in ~\cite{BorLiv} are different. 
Again, by Corollary~\ref{cor: 0 or 1} 
if $V_{0}(K)$ and $V_{1}(K)$ are both nonzero,
then $W$ does not have a smooth spine.

To see that $V_{0}(K)=R_{K}(g(K))\geq 1$ assume the minimum of $R_{K}(g(K))$
is attained for $j_1+\dots +j_n=g(K)$. Then $j_i\ge 1$ for some $i$,
and $R_{K}(g(K))\ge R_{K_i}(j_i)\ge 1$.

To show that $1\le V_{1}(K)=R_{K}(g(K)+1)-1$
assume that the minimum of $R_{K}(g(K)+1)$ is attained for
$j_{1}+\cdots+j_{n}=g(K)+1$. If $j_{i}\geq g(K_{i})+2$, 
then $$R_{K}(g(K)+1)\ge R_{K_{i}}(j_{i})\geq 
R_{K_{i}}(g(K_{i})+2)=V_{-2}(K_i)=V_2(K_i)+2\ge 2$$
as claimed.
Otherwise, there are indices with 
$j_{i}\geq g(K_{i})$ and $j_{l}=g(K_{l})+1$. 
Then $R_{K_{i}}(j_{i})\ge V_{0}(K_{i})\geq 1$ and $R_{K_{l}}(j_{l})=V_{1}(K_{l})+1\geq 1$, and hence $R_{K}(g(K)+1)\geq 2$ as desired. 
\end{proof}

\begin{proof}[Proof of Corollary~\ref{cor: top flat}]
For any $m\geq 2$, there is a topologically slice knot $K_{m}$ with 
$V_{0}(K_{m})=m$~\cite[Proposition 6 and Theorem B.1]{HKL}. 
The corresponding manifold $W$ has a  topologically flat spine.
By Corollary~\ref{cor: 0 or 1} and Proposition~\ref{Vproperty} if $W$ has
a smooth spine, then $V_0(K)\in\{0,1\}$. 
\end{proof}

\small
\bibliographystyle{amsalpha}
\bibliography{spine-8-15-2021}

\end{document}